\documentclass[a4paper,11pt]{article}
\usepackage{amsmath,amsfonts,amsthm}

\newtheorem{prop}{Proposition}[section]
\newtheorem{thm}[prop]{Theorem}
\newtheorem{lemma}[prop]{Lemma}
\newtheorem{defi}[prop]{Definition}
\theoremstyle{remark}
\newtheorem{rmk}[prop]{Remark}
\theoremstyle{definition}

\numberwithin{equation}{section}

\renewcommand{\P}{\mathbb{P}}
\newcommand{\E}{\mathbb{E}}
\newcommand{\F}{\mathfrak{F}}
\renewcommand{\H}{\mathcal{H}}
\newcommand{\fH}{\mathfrak{H}}
\renewcommand{\L}{\mathcal{L}}
\newcommand{\erre}{\mathbb{R}}
\newcommand{\enne}{\mathbb{N}}

\renewcommand{\epsilon}{\varepsilon}
\newcommand{\ip}[2]{\langle#1,#2\rangle}
\newcommand{\bip}[2]{\Big\langle#1,#2\Big\rangle}

\newcommand{\tr}{\mathop{\mathrm{Tr}}\nolimits}

\newsymbol\lesssim 132E

\title{Local well-posedness of Musiela's SPDE with L\'evy noise}

\author{Carlo Marinelli\thanks{Institut f\"ur Angewandte Mathematik,
    Universit\"at Bonn, Wegelerstr. 6, D-53115 Bonn, Germany. URL:
\texttt{http://www.uni-bonn.de/$\sim$cm788}}}

\date{\normalsize April 7, 2007. Revised April 21, 2008 and June 11, 2008.}

\begin{document}
\maketitle

\begin{abstract}
We determine sufficient conditions on the volatility coefficient of
Musiela's stochastic partial differential equation driven by an
infinite dimensional L\'evy process so that it admits a unique local
mild solution in spaces of functions whose first derivative is
square integrable with respect to a weight.
\medskip\par\noindent
\emph{2000 Mathematics Subject Classification:} 60G51, 60H15, 91B28.
\smallskip\par\noindent
\emph{Keywords and phrases:} HJM model, Musiela's stochastic PDE,
stochastic PDEs with jumps, maximal inequalities.
\end{abstract}




\section{Introduction}
The aim of this paper is to find sufficient conditions on the
volatility coefficient of Musiela's stochastic partial differential
equation driven by L\'evy noise such that it admits a unique mild
(local) solution.  In particular, denoting by $u(t,x)$ the forward
rate prevailing at time $t$ with maturity $t+x$, $x\geq 0$, we shall
consider the mild formulation of a stochastic PDE of the form
\begin{equation}
  \label{eq:musso}
du(t,x) = [u_x(t,x) + f(t,x,u^x(t))]\,dt + \ip{\sigma(t,x,u(t,x))}{dM(t)},  
\end{equation}
where $u^x(t):=\{u(t,y):\; y\in[0,x]\}$ and $M$ is a locally square
integrable martingale with independent increments taking values in a
Hilbert space with inner product $\ip{\cdot}{\cdot}$. Moreover,
no-arbitrage conditions uniquely determine the functional form of $f$
in terms of $\sigma$ (see e.g. \cite{BDMKR,JZ,CM-musdet}). The
properties of this SPDE in the case of $M$ being a Wiener process are
rather well studied: let us just mention \cite{filipo} for a
self-contained treatment of existence of solutions, no-arbitrage
conditions, and finite dimensional realizations, \cite{GM-review} for
connections with Kolmogorov equations and option pricing, and
\cite{tehranchi} for the ergodic properties. A basic question to ask
is clearly whether a solution exists to (\ref{eq:musso}), and under
what conditions on $\sigma$ and $f$. If $M$ is a Wiener process, one
can draw on a large body of results (see e.g. \cite{DZ92}) to
establish existence of solutions. In the general case of discontinuous
$M$, the situation turns out to be more involved, not only because the
no-arbitrage ``constraint'' on the drift term $f$ is relatively more
complicated, but also (perhaps mainly) because the theory of
stochastic PDEs driven by jump noise is not so well developed (see
however, for instance, \cite{Gyo-semimg} and \cite{MP} for the
variational approach, \cite{LR-heat} for an analytic approach based on
generalized Mehler semigroups, and \cite{AMR2,PZ-levico} for
the ``mild'' approach).

The literature on HJM models driven by L\'evy processes has
considerably grown in the last few years: let us just cite, among
others, \cite{ALM,EbeJacRai,EbeOez,EbeRai},
and the seminal paper \cite{BDMKR} (where general random measures are
added to a finite dimensional Brownian motion as driving noises). All
these papers essentially take as a starting point the existence of an
HJM model with jumps, and concentrate on deriving no-arbitrage and
completeness properties. Even though several of the papers just
mentioned provide assumptions on the coefficients guaranteeing
existence, an explicit study of sufficient conditions on the data of
the corresponding Musiela stochastic PDE driven by L\'evy noise
guaranteeing (local or global) well-posedness seems to be missing, and
it is the main motivation of the present paper. As mentioned above, we
shall approach this problem studying the (Musiela-type) equation
(\ref{eq:musso}) using the semigroup approach to SPDEs (see e.g.
\cite{DZ92}). Since this approach is mainly developed for the case of
Brownian driving noise, we shall also need to prove some results on
mild solutions of SPDEs driven by L\'evy processes. Our main result is
essentially local existence and uniqueness for (\ref{eq:musso}) in an
appropriate space of absolutely continuous functions, under suitable
regularity and growth assumptions on $\sigma$ and on the L\'evy
measure of the noise $M$.

After this paper was finished, we found a preprint of Filipovi\'c and
Tappe \cite{FilTap} where similar problems are investigated. However,
these authors treat only the case of finitely many independent L\'evy
noises, and obtain solutions in an $L_2$ setting only. The techniques
are also somewhat different, and ours seem to apply to more general
equations as well (we refer in particular to the results of section
\ref{sec:aux}). Moreover, a referee drew our attention to a preprint
of Peszat and Zabczyk \cite{PZ07}, which also deals with existence
for Musiela's SPDE. As in \cite{FilTap}, the L\'evy process is finite
dimensional and solutions are obtained in an $L_2$ setting. However,
in contrast to \cite{FilTap} and the present paper, the state space is
not contained in the space of continuous functions, hence irregular
forward curves seem to be allowed.

The paper is organized as follows: in section \ref{sec:mr} we
introduce a general HJM model driven by an infinite dimensional L\'evy
noise via a Musiela-type SPDE, and we state two existence and uniqueness
results. Section \ref{sec:aux} contains general results on existence
and uniqueness of mild solutions of SPDE with Lipschitz
nonlinearities, which may be of independent interest. Important tools
are maximal inequalities of Burkholder type, which are also proved in
this section. Some more proofs are collected in section
\ref{sec:proofs}, and section \ref{sec:ex} gives two examples covered
by our results.

Let us conclude this introduction with some words about notation.  By
$a \lesssim b$ we mean that there exists a constant $N$ such that $a
\leq Nb$. To emphasize that the constant $N$ depends on a parameter $p$ we
shall write $N(p)$ and $a \lesssim_p b$. Generic constants, which may
change from line to line, are denoted by $N$.  Given two separable
Hilbert spaces $H$, $K$ we shall denote by $\mathcal{L}(K,H)$,
$\mathcal{L}_0(K,H)$, $\mathcal{L}_1(K,H)$ and $\mathcal{L}_2(K,H)$
the space of linear, bounded linear, trace-class, and Hilbert-Schmidt
operators, respectively, from $K$ to $H$.  $\mathcal{L}_1^+$ stands
for the subset of $\mathcal{L}_1$ of positive operators. We shall
write $\mathcal{L}_1(H)$ in place of $\mathcal{L}_1(H,H)$, and
similarly for the other spaces. The Hilbert-Schmidt norm of $F\in\L_2$
is denoted by $|F|_2$. Given a self-adjoint operator
$Q\in\mathcal{L}_1^+(K)$, we denote by $\mathcal{L}_2^Q(K,H)$ the set
of all (possibly unbounded) operators $B:Q^{1/2}K\to H$ such that
$BQ^{1/2}\in\mathcal{L}_2(K,H)$.  The characteristic function of a set
$A$ is denoted by $\chi_A$, and $\chi_r$ stands for the characteristic
function of the set $B_r(H):=\{x\in H:\;|x|\leq r\}$, where $H$ is a
Hilbert space.
We denote by $|\cdot|_p$ the norm of the usual $L_p(\erre_+)$ spaces,
$1\leq p \leq \infty$.  Given a continuously differentiable increasing
function $\alpha:\erre_+\to[1,\infty)$ such that $\alpha^{-1/3}\in
L_1(\erre_+)$, we define $L_{2,\alpha}^n:=L_{2,\alpha}^n(\erre_+)$ as
the space of distributions $\phi$ on $\erre_+$ such that
$$
|\phi|^2_{n,\alpha} := \int_0^\infty |\phi^{(n)}(x)|^2\alpha(x)\,dx
< \infty.
$$
We set $|\cdot|_\alpha:=|\cdot|_{0,\alpha}$. Moreover, the function
$\alpha$ has to be considered fixed once and for all.  For any Hilbert
space $K$, we shall say that $\phi:\erre_+\to K$ belongs to
$L_{2,\alpha}^n(\erre_+,K)$ if
$$
|\phi|^2_{n,\alpha,K} := \int_0^\infty |\phi^{(n)}(x)|_K^2\alpha(x)\,dx
< \infty,
$$
and we shall write $|\cdot|_{\alpha,K}:=|\cdot|_{0,\alpha,K}$.

\noindent We shall use the symbol $D\varphi$ to denote the Fr\'echet derivative
of a function $\varphi$.

\section{Setting and main results}\label{sec:mr}
Let us begin stating precise assumptions on equation (\ref{eq:musso}).
We are given a real separable Hilbert space $K$ and a filtered
probability space $(\Omega,\mathcal{F},\mathbb{F},\P)$,
$\mathbb{F}=(\mathcal{F}_t)_{t\in[0,T]}$, on which a right continuous
square integrable martingale $M$ with values in $K$ is defined.  We
shall denote by $Q \in \L_1^+(K)$ the covariance
operator of $M$, and by
$(e_k)_{k\in\enne}$ a (fixed) orthonormal basis of $K$. For any vector or
function $\phi$ taking values in $K$ we set $\phi^k:=\ip{\phi}{e_k}$.
Moreover, we assume that $M$ has stationary independent increments and
satisfies an exponential integrability condition. In particular, let
$r>0$ be fixed, and assume that there exist positive constants
$\delta>1$ and $C$ such that
\begin{equation}
  \label{eq:mah}
  \E e^{|\ip{\zeta}{M(1)}|} < C \qquad \forall \zeta \in B_{\delta r}(K).
\end{equation}
Let us also define the function
\begin{equation}    \label{eq:pissi}
\psi(\zeta):=\log\E e^{\ip{\zeta}{M(1)}},
\qquad \zeta\in B_{\delta r}(K).
\end{equation}
The function $\psi$ admits the representation
\[
\psi(\zeta) = - \frac12 \ip{R\zeta}{\zeta}
+ \int_K (e^{\ip{\zeta}{x}} - 1 - \ip{\zeta}{x})\,m(dx),
\]
where $R$ and $m$ denote the covariance operator of the Brownian
component and the L\'evy measure of $M$, respectively. Given a real
number $p\geq 2$, we shall write $M \in \mathcal{M}_p$ if $M(0)=0$ and
\[
\int_K |x|^q\,m(dx) < \infty \qquad \forall q\in[2,p].
\]
The (random) function $\sigma:[0,T]\times\erre_+\times\erre \to K$ is
measurable, predictable, and satisfies
$$
\int_0^x \sigma(t,y,u(y))\,dy \in B_r(K) \qquad
\forall (t,x,u(\cdot)) \in [0,T] \times \erre_+ \times L_{2,\alpha}^1
$$
almost surely. 

From now on we shall assume that the measure $\P$ is such that
discounted zero-coupon bond prices are local martingales. Let us
recall that denoting by $B(t,x)$ the price at time $t\geq 0$ of a
zero-coupon bond paying one at time $t+x$ (with $x \geq 0$), the
following identity holds
\[
B(t,x) = \exp\Big( - \int_0^x u(t,y)\,dy\Big).
\]
Defining the mapping
\begin{equation}\label{eq:g}
g(\nu)(t,x) = \bip{\nu(t,x)}{D\psi\big(-\int_0^x\nu(t,y)\,dy\big)}
\end{equation}
on functions $\nu:[0,T]\times\erre_+\to K$ such that $\int_0^x
\nu(t,y)\,dy \in B_r(K)$ for all $(t,x) \in [0,T]\times\erre_+$, the
following functional relationship holds between $f$ and $\sigma$:
$$
f(t,x) = g(\sigma)(t,x).
$$
Note that in (\ref{eq:g}) the Fr\'echet derivative $D\psi$ is well
defined thanks to proposition \ref{prop:dui} below.

\noindent Clearly, writing $\sigma(t,x) \equiv \sigma(t,x,u(x))$ for some
function $u$, the above expression for $f$ is equivalent to
$$
f(t,x,u(\cdot)) = 
\bip{\sigma(t,x,u(x))}{D\psi\big(-\int_0^x\sigma(t,y,u(y))\,dy\big)}.
$$
In particular note that, in view of the previous expression, in
(\ref{eq:musso}) we cannot write, in general, $f(t,x,u(t,x))$, but we
must write instead $f(t,x,u^x(t))$. More regularity assumptions on
$\sigma$, hence on $f$, will be stated or proved as needed.  Finally,
equation (\ref{eq:musso}) is supplemented with an initial condition
$u(0,x)=u_0(x)$, with $u_0$ an $L_{2,\alpha}^1$-valued random variable
measurable with respect to $\mathcal{F}_0$.

We shall rewrite the SPDE (\ref{eq:musso}) as an abstract stochastic
differential equation in the space $H=L_{2,\alpha}^1$, which becomes a
separable Hilbert space when endowed with the inner product
$$
\ip{\phi}{\psi}_H := \int_{\erre_+} \phi'(x)\psi'(x)\,\alpha(x)\,dx
+ \phi(0)\psi(0).
$$
We shall often use the following properties of the space $H$, for a
proof of which we refer to \cite[sect.~5.1]{filipo}. In particular, for
a function $\phi \in H$, we have
\begin{eqnarray}
\label{eq:pippo-inf}
|\phi|_\infty &\lesssim& |\phi|_H,\\
\label{eq:pippo-1}
|(\phi-\phi(\infty))|_1 &\lesssim& |\phi|_H,\\
\label{eq:pippo-4}
|(\phi-\phi(\infty))^2|_\alpha &\lesssim& |\phi|_H^2,
\end{eqnarray}
where the constants depend only on $\alpha$.
This state space has apparently been first used by
Filipovi\'c \cite{filipo}. Other authors have considered
related SPDEs in different function spaces, e.g.  in weighted $L_2$
spaces, weighted Sobolev spaces, or fractional Sobolev spaces (see
\cite{GM-review}, \cite{vargiolu}, \cite{EkeTaf} respectively).

Let us define on $H$ the operator $A:\phi \mapsto \phi'$, with domain
$D(A)=L^1_{2,\alpha}\cap L^2_{2,\alpha}$, which generates the
semigroup of right shifts $e^{tA}\phi(x):=\phi(x+t)$, $t\geq 0$.
We can formally rewrite Musiela's SPDE (\ref{eq:musso}) in the following
abstract form:
\begin{equation}
\label{eq:ee}
du(t) = [Au(t) + f(t,u(t))]\,dt + B(t,u(t))\,dM(t),
\quad u(0)=u_0,
\end{equation}
where $f:[0,T] \times H \to H$ is defined by $x \mapsto
g(\sigma)(t,x)$ (with a slight but harmless abuse of
notation), and $B:[0,T] \times H \to \mathcal{L}(K,H)$ is defined by
$[B(t,u)\phi](x) = \ip{\sigma(t,x,u(x))}{\phi}_K$, $\phi\in K$.
This rewriting is, for the time being, only formal, because one needs
to prove that the images of $f$ and $B$ are actually contained in $H$
and $\L_Q^2(K,H)$, which is not true without further assumptions. For
now, however, we content ourselves with the formal equation
(\ref{eq:ee}), noting that it can be made rigorous under
conditions that will be given in theorems \ref{thm:1} and \ref{thm:2}
below.

The concept of solution we shall work with and the space where mild
solutions are sought are defined next.
\begin{defi}
  A predictable process $u:[0,T]\times\Omega \to H$ is a mild solution
  of (\ref{eq:ee}) if it satisfies
  $$
  u(t) = e^{tA}u_0 + \int_0^t e^{(t-s)A}f(s,u(s))\,ds + \int_0^t
  e^{(t-s)A} B(s,u(s))\,dM(s)
  $$
  $\P$-a.s. for almost all $t\in[0,T]$.
\end{defi}
\begin{defi}
  We shall denote by $\mathcal{H}_p(T)$, $p\geq 2$, the space of all
  predictable processes $u:[0,T]\to H$ such that
  $$
  |[u]|_p^p := \sup_{t\in[0,T]} \E|u(t)|^p <\infty.
  $$
\end{defi}
\noindent We have the following local well-posedness result in $\H_2(T)$.
\begin{thm}\label{thm:1}
  Assume that
  \begin{enumerate}
  \renewcommand{\labelenumi}{{\rm (\roman{enumi})}}
  \item $\sigma^k \in C^{0,1,2}([0,T] \times \erre_+ \times \erre)$ for
  all $k\in\enne$;
  \item there exists $\beta \in L_{2,\alpha}(\erre_+,\ell_2)$ such that,
  for all $k\in\enne$,
  \begin{equation}
    \label{eq:pru}
  |\sigma^k_x(t,x,u) - \sigma^k_x(t,x,v)| \leq \beta^k(x)|u-v|
  \end{equation}
  for almost all (a.a.) $(t,x) \in [0,T] \times \erre_+$, $\P$-a.s.;
  \item there exists $\gamma=(\gamma^k)_{k\in\mathbb{N}} \in \ell_2$
such that
  \begin{equation}
    \label{eq:gna}
  |\sigma^k_u(t,x,u)| + |\sigma^k_{uu}(t,x,u)| \leq \gamma^k  
  \end{equation}
  for a.a. $(t,x,u) \in [0,T] \times \erre_+ \times \erre$, $\P$-a.s.;
\item there exists a constant $C$ such that
  $|\sigma(t,\cdot,0)|_{1,\alpha,K} < C$ for a.a. $t \in [0,T]$,
  $\P$-a.s.;
\item $\E|u_0|^2 < \infty$.
  \end{enumerate}
  Then (\ref{eq:musso}) has a unique local mild solution in $\H_2(T)$.
  Furthermore, the solution depends Lipschitz continuously on the
  initial datum.
\end{thm}
\begin{rmk}
  In (\ref{eq:gna}) one could replace the second term on the left hand
  side with $[\sigma_u^k(t,x,u)]_{\mathrm{Lip}}$, where
  $[\cdot]_{\mathrm{Lip}}$ stands for the Lipschitz norm with respect
  to the third variable. In other words, the same proof goes through if
  one replaces the assumption on boundedness of the second derivative
  of $\sigma^k$ with respect to $u$ with an assumption of Lipschitz
  continuity of $\sigma^k_u$ with uniform Lipschitz constant.
\end{rmk}

The proof of this theorem (which is given in section \ref{sec:proofs})
relies on a more general result on existence and uniqueness for
stochastic evolution equations with Lipschitz nonlinearities (see
theorem \ref{thm:mild1} below), combined with regularity results for
$f$.

The above theorem, as well as theorem \ref{thm:mild1} below, are
stated and proved for their simplicity, even though a more refined
result holds true. In fact, one can look for mild solutions of
(\ref{eq:ee}) in smaller spaces that are the infinite dimensional
version of the spaces usually considered in the theory of finite
dimensional SDEs (see e.g. \cite{Protter}).
\begin{defi}
  We shall denote by $\mathbb{H}_p(T)$, $p \geq 2$, the space of all
  predictable processes $u:[0,T]\to H$ such that
  $$
  \|u\|_p^p := \E\sup_{t\in[0,T]} |u(t)|^p < \infty.
  $$
\end{defi}
\noindent The following local well-posedness result in
$\mathbb{H}_p(T)$ is the main result of this paper.
\begin{thm}\label{thm:2}
  Let $p\geq 2$ and assume that hypotheses (i)--(iv) of theorem
  \ref{thm:1} hold. Moreover, assume that $\E|u_0|^p < \infty$ and
  $M\in\mathcal{M}_p$. Then (\ref{eq:musso}) has a unique local mild
  solution in $\mathbb{H}_p(T)$. Furthermore, the solution depends
  Lipschitz continuously on the initial datum.
\end{thm}

\section{Auxiliary results}\label{sec:aux}
Let us start proving a regularity result for the function $\psi$
defined in (\ref{eq:pissi}) above, under the assumption
(\ref{eq:mah}).
\begin{prop}    \label{prop:dui}
  One has $\psi \in C^2_b(B_r(K))$ and $D^2\psi$ is Lipschitz on
  $B_r(K)$.
\end{prop}
\begin{proof}
We claim that the G\^ateaux derivative of $\psi$ at a point
$\zeta\in B_r(K)$ is given by
\begin{equation}     \label{eq:dpsi}
D_G\psi(\zeta): \phi \mapsto \ip{R\zeta}{\phi}
+ \int_K \ip{x}{\phi} (e^{\ip{\zeta}{x}}-1)\,m(dx).    
\end{equation}
In fact, for any $\zeta \in B_r(K)$ and $\xi\in K$, we have
\begin{eqnarray*}
\lim_{\varepsilon\to 0}
\frac{\psi(\zeta+\varepsilon\xi)-\psi(\zeta)}{\varepsilon} &=&
\ip{R\zeta}{\xi}
+ \lim_{\varepsilon\to 0} \int_K \big( e^{\ip{\zeta}{x}}
  \frac{e^{\varepsilon\ip{\xi}{x}}-1}{\varepsilon}
  + \ip{\xi}{x} \big)\,m(dx)\\
&=& \ip{R\zeta}{\xi}
    + \int_K \ip{\xi}{x} (e^{\ip{\zeta}{x}}-1)\,m(dx).    
\end{eqnarray*}
The passage to the limit under the integral sign in the above
expression is possible thanks to the dominated convergence theorem. In
fact, by the mean value theorem, one has
\[
e^{\ip{\zeta}{x}} \frac{e^{\varepsilon\ip{\xi}{x}}-1}{\varepsilon} 
= e^{\ip{\zeta}{x}} e^{\varepsilon\eta\ip{\xi}{x}} \ip{\xi}{x}
\]
for some $\eta\in(0,1)$, hence
\[
\Big| e^{\ip{\zeta}{x}} \frac{e^{\varepsilon\ip{\xi}{x}}-1}{\varepsilon}
\Big| \leq |\xi| \, |x| \, e^{|\ip{\zeta}{x}|} \,
             e^{\varepsilon|\ip{\xi}{x}|},
\]
and, setting $\delta'=\delta/(\delta-1)$ and appealing to H\"older's
inequality,
\begin{eqnarray*}
\lefteqn{
\int_K |x| \, e^{|\ip{\zeta}{x}|} \, e^{\varepsilon|\ip{\xi}{x}|}\,m(dx)}\\
&\leq&
\Big( \int_K e^{|\ip{\delta\zeta}{x}|} \,m(dx) \Big)^{\frac{1}{\delta}}
\Big( \int_K |x|^{2\delta'}\,m(dx) \Big)^{\frac{1}{2\delta'}}
\Big( \int_K e^{2\varepsilon\delta'|\ip{\xi}{x}|}\,m(dx) \Big)^{\frac{1}{2\delta'}}\\
&<& \infty,
\end{eqnarray*}
for all $\varepsilon < (2\delta'|\xi|)^{-1}\delta r$.

\noindent Let us now show that $D_G\psi$ is continuous from $K$ to
$\mathcal{L}(K,\erre)$. Let $\zeta_n \to \zeta$ as $n \to \infty$,
$h\in K$. Then
\begin{eqnarray*}
\sup_{|h|\leq 1} \big| D_G\psi(\zeta_n)h - D_G\psi(\zeta)h \big| &\leq&
\sup_{|h|\leq 1} \int_K |\ip{x}{h}| \,
    \big|e^{\ip{\zeta_n}{x}} - e^{\ip{\zeta}{x}}\big|\,m(dx)\\
&\leq& \int_K |x| \big|e^{\ip{\zeta_n}{x}} - e^{\ip{\zeta}{x}}\big|\,m(dx)\\
&\to& 0
\end{eqnarray*}
as $n \to \infty$ by the dominated convergence theorem, with
computations completely analogous to the above ones.
Since $D_G\psi$ is continuous, then $D\psi=D_G\psi$ by a
well-known criterion for Fr\'echet differentiability (see e.g.
\cite[thm. 1.9]{AmbPro}).

Proceeding in a completely similar way one can prove that $\psi$ is
twice Fr\'echet differentiable in $B_r(K)$, with
$$
D^2\psi(\zeta): (\phi,\eta) \mapsto \int_K \ip{x}{\phi} \ip{x}{\eta}
 e^{\ip{x}{\zeta}}\,m(dx).
$$
The boundedness of $D\psi$ on $B_r(K)$ can again be proved by applying 
H\"older's inequality to (\ref{eq:dpsi}).

In a completely analogous way one can prove that $D^2\psi$ as well as
$D^3\psi$ are bounded on $B_r(K)$, in particular $D^2\psi$ is
Lipschitz.
\end{proof}
\begin{rmk}
  Simpler but less general assumptions than (\ref{eq:mah}) may be
  given. In fact, (\ref{eq:mah}) is clearly satisfied if $\int
  e^{\delta r |x|}\,m(dx)<\infty$.
  Another sufficient
  condition for (\ref{eq:mah}) is that the map $\varphi: \zeta \mapsto
  \int_K e^{|\ip{x}{\zeta}|}\,m(dx)$ is continuous with respect to the
  weak topology of $K$. In fact, by Banach-Alaoglu's theorem, the ball
  $B_{\delta r}$ is weakly compact in $K$, hence the image of
  $B_{\delta r}$ under $\varphi$ is compact in the real line and
  $\varphi$ attains a (finite) maximum at some point $\zeta_0 \in
  B_{\delta r}$.
\end{rmk}

\medskip

Let us denote by $[M]$ and $\langle M \rangle$ the quadratic variation
and the Meyer process of $M$, respectively.  We shall need the
following elementary lemma, whose proof is included for completeness.
\begin{lemma}
\label{lem:cov}
There exists a (deterministic) operator $Q\in \L_1^+(K)$ such that
$\langle M\rangle(t) = t\,\tr(Q)$ for all $t\geq 0$.
\end{lemma}
\begin{proof}
Let $Q$ be the correlation operator of $M(1)$, i.e. the operator defined by
\[
\ip{Qx}{y} = \E\ip{M(1)}{x} \ip{M(1)}{y},
\qquad x,\; y \in K.
\]
Recalling that $\E|M(1)|^2<\infty$, it is immediate to prove that
$Q\in\mathcal{L}_1^+(K)$.
Since $M$ has homogeneous independent increments, we also have
\[
\E\big[|M(t)|^2 - |M(s)|^2\big|\mathcal{F}_s\big]
= \E\big[|M(t)-M(s)|^2\big]
= (t-s)\tr(Q),
\]
where for the second equality we have used the facts that $M(t)-M(s) =
M(t-s)$ in distribution, $\E|M(t)|^2=t\E|M(1)|^2$, and
$\E|M(1)|^2=\tr(Q)$, as follows by definition of $Q$.
\end{proof}

Let $E$ be a real separable Hilbert spaces, and $A:D(A)\subset E
\to E$ the generator of a strongly continuous semigroup $e^{tA}$ on
$E$.
We will extensively use the following isometric formula due to
M\'etivier and Pistone \cite{MP-Z}: let $Q$ be the correlation
operator of $M$, and $F:[0,T] \to \L_Q^2(K,E)$ a predictable process
such that
$$
\E\int_0^T |F(s)|_Q^2\,ds < \infty.
$$
Then we have
\begin{equation}
  \label{eq:isom}
\E \Big| \int_0^t F(s)\,dM(s) \Big|^2 =
\E \int_0^t |F(s)|^2_Q\,ds.
\end{equation}
\begin{rmk}
  In fact the isometric formula of \cite{MP-Z} states that for every
  square integrable martingale $M$ and $F$ as above such that
  $$
  \E \int_0^t |F(s)|^2_{Q_M(s)}\,d\langle M \rangle(s) < \infty \qquad
  \forall t \leq T,
  $$
  one has
  \begin{equation}\label{eq:iso-vera}
    \E \Big| \int_0^t F(s)\,dM(s) \Big|^2 =
    \E \int_0^t |F(s)|_{Q_M(s)}^2 \,d\langle M \rangle(s),
  \end{equation}
  where $Q_M:[0,T] \to \L_1^+(K)$ is a predictable process with
  $\tr(Q_M(s))=1$ a.s., for all $s\leq T$. If $M$ has independent
  increments, then one also has $Q_M(t) = Q/\tr(Q)$, hence
  (\ref{eq:isom}) follows immediately by (\ref{eq:iso-vera}) and
  lemma \ref{lem:cov}.
\end{rmk}

In this section we establish existence and uniqueness of mild
solutions of equations of the form (\ref{eq:ee}) on a general Hilbert
space $E$, assuming only $A$ generates a strongly continuous
semigroup (resp. strongly continuous contraction semigroup) and the
nonlinear terms $f$ and $B$ are maps satisfying a Lipschitz
assumptions.

In the following theorem we shall denote by
$\Psi:L_2(\Omega,\mathcal{F},\P;H) \to \H_2(T)$ the solution map of
(\ref{eq:ee}), that is, $\Psi(\eta)$ stands for the solution of
(\ref{eq:ee}) with initial condition $\eta$. To simplify notation, we
shall set $\H_p=L_p(\Omega,\mathcal{F},\P;H)$ for $p\geq 2$.
\begin{thm}
\label{thm:mild1}
Assume that $\E|u_0|^2<\infty$, $e^{sA}B(t,x) \in
\mathcal{L}^Q_2(K,E)$ for all $(t,x)\in[0,T]\times E$, and there exists $h\in
L_{2,loc}(\erre)$ such that, for all $s$, $t \in [0,T]$,
  \begin{equation}
    \label{eq:i}
    |e^{sA}f(t,x)| + |e^{sA}B(t,x)|_Q \leq h(s)(1+|x|), \quad
  \end{equation}
  \begin{equation}
    \label{eq:ii}
    |e^{sA}(f(t,x) - f(t,y))| + |e^{sA}(B(t,x) - B(t,y))|_Q \leq h(s)|x-y|.
  \end{equation}
  for all $x$, $y \in E$, $\P$-a.s..  Then equation (\ref{eq:ee})
  admits a unique mild solution in $\mathcal{H}_2(T)$. Moreover, there
  exists a constant $N$, independent of $u_0$, $v_0$, such that
  \begin{equation}
    \label{eq:leone}
    |[\Psi(u_0)-\Psi(v_0)]|_2 \leq N|u_0-v_0|_{\H_2}
  \end{equation}
  for all $u_0$, $v_0 \in \H_2$.
\end{thm}
\begin{proof}
  The proof follows the same reasoning used for equations with Wiener
  noise. Some redundant details are therefore omitted. We shall prove
  that the mapping $\F: \H_2(T) \to \H_2(T)$ defined by
\begin{equation}
\label{eq:fi}
\F u(t) =  e^{tA} u_0 
+ \int_0^t e^{(t-s)A} f(s,u(s))\,ds
+ \int_0^t e^{(t-s)A} B(s,u(s))\,dM(s)
\end{equation}
is well defined and is a contraction, after which the result will
follow easily by Banach's fixed point theorem.  Let us first prove
that the image of $\F$ is in fact contained in $\H_2(T)$: to
this purpose, we have to show that, for any $u\in \H_2(T)$,
$\F u$ admits a predictable modification and that
$|[\F u]|_2<\infty$. Predictability of $\F u$ follows
by the mean-square continuity of the stochastic convolution term in
(\ref{eq:fi}): in fact, setting $M_A(t)=\int_0^t e^{(t-s)A}
B(s,u(s))\,dM(s)$, a simple calculation shows that, for $0\leq s\leq
t \leq T$,
\begin{eqnarray*}
\E|M_A(t)-M_A(s)|^2 &\lesssim&
\E \Big| \int_0^s (e^{(t-r)A}-e^{(s-r)A})B(r,u(r))\,dM(r) \Big|^2 \\
&& + \E \Big| \int_s^t e^{(t-r)A}B(r,u(r))\,dM(r) \Big|^2\\
&\leq& \E \int_0^s |e^{(t-r)A}-e^{(s-r)A}|^2|B(r,u(r))|_Q^2\,dr \\
&& + \E \int_s^t |e^{(t-r)A}|^2 |B(r,u(r))|_Q^2\,dr,
\end{eqnarray*}
where the second inequality follows by the isometric formula for
stochastic integrals. Therefore $\E|M_A(t)-M_A(s)|^2 \to 0$ as $s\to
t$.

\noindent Moreover, we have
\begin{eqnarray*}
|[\F u]|^2_2 &\lesssim& \sup_{t\leq T} \E|e^{tA} u_0|^2
+ \sup_{t\leq T} \E\Big| \int_0^t e^{(t-s)A} f(s,u(s))\,ds \Big|^2\\
&& + \sup_{t\leq T} \E\Big| \int_0^t e^{(t-s)A} B(s,u(s))\,dM(s) \Big|^2,
\end{eqnarray*}
and $\sup_{t\leq T} \E|e^{tA}u_0|^2 \leq \sup_{t\leq T} |e^{tA}|^2
\E|u_0|^2<\infty$. Cauchy-Schwarz' inequality and
(\ref{eq:i}) yield
\begin{eqnarray*}
\E\left|\int_0^t e^{(t-s)A}f(s,u(s))\,ds\right|^2 &\leq&
2T \E\int_0^t h^2(t-s)(1+|u(s)|^2)\,ds,
\end{eqnarray*}
hence
$$
\sup_{t\leq T} \E\left|\int_0^t e^{(t-s)A}f(s,u(s))\,ds\right|^2 
\leq N \sup_{t\leq T} \E(1+|u(t)|^2) < \infty,
$$
where $N=2T\int_0^T h^2(s)\,ds$.

Similarly, using the isometric formula (\ref{eq:isom}), lemma
\ref{lem:cov} and (\ref{eq:i}), we have
\begin{eqnarray*}
\E\left| \int_0^t e^{(t-s)A}B(s,u(s))\,dM(s) \right|^2 &=&
\E\int_0^t \big|e^{(t-s)A}B(s,u(s))Q^{1/2}\big|^2_2\,ds \\
&\leq& \E\int_0^t h(t-s)^2 (1+|u(s)|)^2\,ds.
\end{eqnarray*}
Then
\begin{eqnarray*}
\sup_{t\leq T} \E\left| \int_0^t e^{(t-s)A}B(s,u(s))\,dM(s) \right|^2 &\leq&
N \sup_{t\leq T} \E(1+|u(t)|^2),
\end{eqnarray*}
with $N=2\int_0^T h^2(s)ds$.
Therefore we have
$$
|[\F u]|_2^2 = \sup_{t\leq T} \E|\F u(t)|^2 \lesssim
1 + \sup_{t\leq T} \E|u(t)|^2 < \infty,
$$
which proves that $\F u \in \H_2(T)$.

Completely analogous calculations involving (\ref{eq:ii}) instead of
(\ref{eq:i}) show that
$$
|[\F u-\F v]|^2_2
\leq N|[u-v]|^2_2,
$$
with
$$
N = N(T,h) = 2(T+1)\int_0^T h^2(s)\,ds.
$$
Since $h\in L_{2,loc}(\erre)$, one can find $T_0>0$ so that $N(T_0,h)<1$,
thus one obtains, by Banach's fixed point theorem, existence and
uniqueness of a mild solution to (\ref{eq:ee}) on the time interval
$[0,T_0]$. Then one can proceed with classical extension arguments,
proving that a global solution on the time interval $[0,T]$ exists and
is unique.

In order to prove (\ref{eq:leone}), it is convenient to regard the map
$\mathfrak{F}$ as a function from $\H_2 \times \H_2(T)$ to $\H_2(T)$.
Let us assume, without loss of generality, that $T$ is such that
$|[\F(u_0,u)-\F(u_0,v)]|_2 \leq N_1 |[u-v]|_2$ with $N_1<1$. Then the
above fixed-point argument implies that the solution map
$\Psi:\H_2\to\H_2(T)$ is such that $\F(u_0,\Psi(u_0))=\Psi(u_0)$.
Moreover, by definition of $\F$, we have that
$$
|[\F(u_0,w) - \F(v_0,w)]|_2 \leq N_2 |u_0-v_0|_\H \qquad
\forall w \in \H_2(T),
$$
with $N_2=\big(\sup_{t\leq T} |e^{tA}|\big)^{1/2}$.
Therefore we can write
\begin{eqnarray*}
|[\Psi(u_0)-\Psi(v_0)]|_2 &=& |[\F(u_0,\Psi(u_0)) - \F(u_0,\Psi(u_0))]|_2\\
&\leq& |[\F(u_0,\Psi(u_0)) - \F(u_0,\Psi(v_0))]|_2 \\
&&     + |[\F(u_0,\Psi(v_0)) - \F(v_0,\Psi(v_0))]|_2\\
&\leq& N_1 |[\Psi(u_0)-\Psi(v_0)]|_2 + N_2 |u_0-v_0|_{\H_2},
\end{eqnarray*}
hence $|[\Psi(u_0)-\Psi(v_0)]|_2 \leq (1-N_1)^{-1}N_2 |u_0-v_0|_{\H_2}$.
\end{proof}

\begin{rmk}
  In order to obtain existence and uniqueness in $\H_p(T)$, $p>2$, one
  would need an estimate for a term of the type
  $$
  \sup_{t\leq T} \E \Big| \int_0^t e^{(t-s)A}F(s)\,dM(s) \Big|^p
  $$
  with $F$ a predictable function, whereas for the proof in $\H_2(T)$
  the isometric property of the stochastic integral is enough. An
  estimate of Burkholder type in $\H_p(T)$, $p>2$, for stochastic
  convolutions with respect to compensated Poisson measures is
  announced in \cite{Knoche-CRAS}.
  The maximal inequality (\ref{eq:BDGconv}) seems to
  be simpler and more natural, even though it holds only for
  pseudo-contraction semigroups $e^{tA}$.
\end{rmk}

The main result of this section is the following theorem, where the
solution map $\Psi$ is now defined from
$\H_p=L_p(\Omega,\mathcal{F},\P;H)$ to $\mathbb{H}_p(T)$.
\begin{thm}
\label{thm:mild2}
Let $p \geq 2$. Assume that $e^{tA}$ is a contraction semigroup,
$\E|u_0|^p<\infty$, $M\in\mathcal{M}_p$, and there exists $h\in
L_{p,loc}(\erre)$ such that
\begin{equation}\label{eq:i'}
  |e^{sA}f(t,x)| + |B(s,x)| \leq h(s)(1+|x|),
\end{equation}
\begin{equation}\label{eq:ii'}
  |e^{sA}(f(t,x) - f(t,y))| + |B(s,x) - B(s,y)| \leq h(s)|x-y|.
\end{equation}
for all $x$, $y \in E$. Then equation (\ref{eq:ee}) admits a unique
mild solution in $\mathbb{H}_p(T)$. Moreover, there exists a constant
$N$, independent of $u_0$, $v_0$, such that
\begin{equation}
  \label{eq:fuoco}
  \|\Psi(u_0)-\Psi(v_0)\|_p \leq N|u_0-v_0|_{\H_p}
\end{equation}
for all $u_0$, $v_0 \in \H_p$.
\end{thm}
Note that in order to obtain a solution in $\mathbb{H}_p(T)$, which is
a subset of $\mathcal{H}_p(T)$, one has to assume that $A$ generates a
contraction semigroup, which was not assumed in theorem
\ref{thm:mild1}.

In order to prove theorem \ref{thm:mild2} we need to establish a
maximal inequality of Burkholder type for stochastic convolutions,
which may be of independent interest. For related estimates, which
could probably be used in this context as well, see also \cite{HZ,
Ichi, Kote-sub, Kote-Doob}.

Let us first recall the Burkholder inequality for Hilbert space valued
martingales, see e.g. \cite[{\S}6.E.3]{Met}.
\begin{prop}
  Let $X$ be a $K$-valued local martingale with $X(0)=0$. For every
  $p\in [2,\infty)$ there exists a constant $C(p)$ depending on $p$
  only such that
  \begin{equation}
    \label{eq:BDG}
    \E \sup_{t\leq T} |X(t)|^p \leq C(p) \E [X](T)^{p/2}.
  \end{equation}
\end{prop}
For $F$ predictable and locally bounded, setting $X(t) = \int_0^t
F(s)\,dM(s)$, (\ref{eq:BDG}) implies
\begin{equation}     \label{eq:bagigio}
\E \sup_{t\leq T} \Big| \int_0^t F(s)\,dM(s) \Big|^p
\leq C(p) \E \Big( \int_0^T |F(s)|^2 \,d[M](s) \Big)^{p/2}.
\end{equation}
This estimate allows to obtain a maximal inequality of Bichteler-Jacod
type in Hilbert space, following very closely the proof of lemma 4.1
in \cite{ProTal-Euler}. That is, we have the following lemma, of which
we include a proof for the reader's convenience.
\begin{lemma}
  Let $p\geq 2$. Let $Y$ be a martingale L\'evy process with $Y(0)=0$
  and $\int_K |x|^q\,m(dx)<\infty$ for all $q\in [2,p]$, where $m$ is
  the L\'evy measure of $Y$.  Let $F$ be a predictable locally bounded
  process such that $\E\int_0^T |F(s)|^p\,ds<\infty$. Then there
  exists a constant $N$ which depends only on $p$ and $T$ such that
  \begin{multline}
    \label{eq:BJ}
    \E\sup_{t\leq T} \Big| \int_0^t F(s)\,dY(s) \Big|^p \\ \leq N
    \bigg[|R|_1^{p/2} + \int_K |x|^p\,m(dx) + \Big(\int_K |x|^2\,m(dx) \Big)^{p/2} \bigg]
    \, \E\!\int_0^T |F(s)|^p\,ds,
  \end{multline}
  where $R$ denotes the covariance operator of the Brownian
  component of $Y$.
\end{lemma}
\begin{proof}
Appealing to the L\'evy-It\^o's decomposition and Burkholder's inequality
for stochastic integrals with respect to Hilbert space valued Wiener processes,
we can assume, without loss of generality, that $R=0$.
Let $k\in\mathbb{N}$ be such that
$p\in[2^k,2^{k+1}[$. Then Burkholder's inequality (\ref{eq:bagigio}) yields
\[
\E\Big|\sup_{t\leq T} \int_0^t F(s)\,dY(s)\Big|^p \lesssim_p 
\E\Big( \int_0^T |F(s)|^2\,d[Y](s)\Big)^{p/2}.
\]
Setting
\[
\alpha = \E[Y](1) = \E \sum_{s\leq 1} |\Delta Y(s)|^2
= \int_K |x|^2\,m(dx),
\]
we have that $\alpha<\infty$ and $[Y](t)-\alpha t$ is a martingale,
because $[Y](t)$ is a one-dimensional L\'evy process. Moreover,
we can write
\begin{multline*}
\E\Big( \int_0^T |F(s)|^2\,d[Y](s) \Big)^{p/2} \\
\lesssim_p
\E\Big| \int_0^T |F(s)|^2\,d([Y](s)-\alpha s) \Big|^{p/2} +
\alpha^{p/2} \E\Big( \int_0^T |F(s)|^2\,ds \Big)^{p/2}.
\end{multline*}
Applying again Burkholder's inequality to the first term on the right
hand side of the previous expression one gets
\[
\E\Big|\int_0^T |F(s)|^2\,d([Y](s)-\alpha s)\Big|^{p/2} \lesssim_p
\E\Big(\sum_{s\leq t}|F(s)|^4 |\Delta Y(s)|^4\Big)^{p/4}.
\]
Since
\[
\Big( \int_0^T |F(s)|^2\,ds \Big)^{p/2} \lesssim_{p,T}
\int_0^T |F(s)|^p\,ds,
\]
we obtain, iterating the above decomposition procedure,
\begin{align*}
\E \Big| \sup_{t\leq T} \int_0^t F(s)\,dY(s) \Big|^p
\;\lesssim_{p,T}\;& 
\E \Big( 
\sum_{s\leq t} |F(s)|^{2^{k+1}} |\Delta Y(s)|^{2^{k+1}}
\Big)^{p/2^{k+1}}\\
& + \left(
\sum_{i=1}^k \Big( \int |x|^{2^i}\,m(dx)\Big)^{p/2^i}
\right)
\E \int_0^T |F(s)|^p\,ds.
\end{align*}
We can write
\begin{align*}
\Big( \sum_{s\leq t} |F(s)|^{2^{k+1}} |\Delta Y(s)|^{2^{k+1}}
\Big)^\frac{p}{2^{k+1}} &=
\left( \sum_{s\leq t} \big( |F(s)|^{2^k} |\Delta Y(s)|^{2^k} \big)^2
\right)^{\frac12 \frac{p}{2^k}}\\
&=: |a|_{\ell_2}^{p/2^k} \leq |a|_{\ell_{p/2^k}}^{p/2^k} =
\sum_{s\leq t} |F(s)|^p |\Delta Y(s)|^p,
\end{align*}
where we have used the inequality $|a|_{\ell_2} \leq |a|_{\ell_q}$,
which holds for all $a \in \ell_q$, $q\in [1,2]$.

\noindent Since the process $|F(\cdot)|^p$ is predictable and
$\sum_{s\leq t} |\Delta Y(s)|^p$ is an increasing adapted c\`adl\`ag
process with compensator $t \mapsto t\int_K |x|^p\,m(dx)$, we have that
\[
\int_0^t |F(s)|^p\,d\Big( \sum_{r\leq s} |\Delta Y(r)|^p
- s \int_K |x|^p\,m(dx) \Big)
\]
is a martingale with expectation zero. Therefore we can write
\begin{equation*}
\E \Big| \int_0^t F(s)\,dY(s) \Big|^p \;\lesssim_{p,T}\;
\bigg[
\int_K |x|^p\,m(dx)
+ \sum_{i=1}^k \Big( \int_K |x|^{2^i}\,m(dx)\Big)^{p/2^i}
\bigg]
\int_0^t |F(s)|^p\,ds
\end{equation*}
In order to complete the proof it is then enough to show that
\[
\Big( \int_K |x|^{2^i}\,m(dx)\Big)^{p/2^i}
\leq
\Big( \int_K |x|^2\,m(dx)\Big)^{p/2} +
\int_K |x|^p\,m(dx)
\]
for all $i\in[1,k]$.
Define the probability measure on $K$
\[
\mu(dx) = \frac{1}{Z} |x|^2 m(dx),
\]
where $Z=\int_K |x|^2\,m(dx)$, and set, for simplicity of
notation, $q=2^i$.
Then proving the above inequality is equivalent to proving that
\[
Z^{p/q} \Big( \int_K |x|^{q-2}\,\mu(dx)\Big)^{p/q}
\leq
Z^{p/2} + Z \int_K |x|^{p-2}\,\mu(dx).
\]
We distinguish two cases: if
\[
\Big( \int_K |x|^{q-2}\,\mu(dx)\Big)^{p/q}
\leq
Z^{p/2-p/q}
\]
the inequality is obvious, otherwise, if
\[
Z \leq \Big( \int_K |x|^{q-2}\,\mu(dx)\Big)^{\frac{2}{q-2}},
\]
then it is sufficient to prove that
\[
Z^{p/q-1}\Big( \int_K |x|^{q-2}\,\mu(dx)\Big)^{p/q}
\leq
\int_K |x|^{p-2}\,\mu(dx),
\]
which follows immediately by the upper bound on $Z$ and
Jensen's inequality.
\end{proof}
\begin{rmk}
  The fact that one can obtain a Bichteler-Jacod inequality with
  constant independent of the dimension is irrelevant for the purposes
  of \cite{ProTal-Euler}, which deals with stochastic equations in
  $\erre^d$, but appears to be known to specialists in Malliavin
  calculus for jumps -- see e.g.  lemma 5.1 and remark 5.2 in
  \cite{BGJ}, and references therein.
\end{rmk}
For convenience we shall denote by $\mathfrak{m}_p$ the term in square
parentheses appearing in (\ref{eq:BJ}). The same type of inequality
can be established also for stochastic convolutions, even though in
general they are not martingales.
\begin{prop}
  Let $A$ be the generator of a strongly continuous contraction
  semigroup on $E$ and $F:[0,T] \to \mathcal{L}_2(K,E)$ be a locally
  bounded predictable process. Let $p\in [2,\infty)$ and $M \in
  \mathcal{M}_p$. Then there exists a constant $N$, which depends only
  on $p$ and $T$, such that
  \begin{equation}
    \label{eq:BDGconv}
    \E\sup_{t\leq T} \Big| \int_0^t e^{(t-s)A} F(s)\,dM(s)\Big|^p
    \leq N \mathfrak{m}_p \, \E \int_0^T |F(s)|^p\,ds.
  \end{equation}
\end{prop}
\begin{proof}
  We shall follow the approach of \cite{HauSei}. In particular, by
  Sz.-Nagy's theorem on unitary dilations, there exists a Hilbert
  space $\bar E$, with $E$ isometrically embedded into $\bar E$, and a
  unitary strongly continuous group $T(t)$ on $\bar E$ such that $\pi
  T(t)x=e^{tA}x$ for all $x\in E$, $t\in \erre$, where $\pi$ denotes
  the orthogonal projection from $\bar E$ to $E$. Then we have,
  recalling that the operator norms of $\pi$ and $T(t)$ are less or
  equal than one,
  \begin{eqnarray*}
    \E\sup_{t\leq T} \Big| \int_0^t e^{(t-s)A} F(s)\,dM(s)\Big|_E^p
    &=& \E\sup_{t\leq T}
             \Big|\pi T(t) \int_0^t T(-s) F(s)\,dM(s)\Big|_{\bar E}^p\\
    &\leq& |\pi|^p \; \sup_{t\leq T} |T(t)|^p \;
           \E\Big|\int_0^t T(-s) F(s)\,dM(s)\Big|_{\bar E}^p\\
    &\leq& \E\sup_{t\leq T} \Big|\int_0^t T(-s) F(s)\,dM(s)\Big|_{\bar E}^p.\\
  \end{eqnarray*}
  Since the integral in the last expression is a martingale,
  inequality (\ref{eq:BJ}) implies that there exists a constant $N=N(p,T)$
  such that
  \begin{eqnarray*}
    \E\sup_{t\leq T} \Big| \int_0^t e^{(t-s)A} F(s)\,dM(s) \Big|_E^p &\leq&
    N \mathfrak{m}_p \, \E\int_0^T |T(-s)F(s)|^p\,ds \\
    &\leq& N \mathfrak{m}_p \, \E\int_0^T |F(s)|^p\,ds,
  \end{eqnarray*}
  where we used again that $T(t)$ is a unitary group and that the
  norms of $F(s)$ in $\mathcal{L}_0(K,\bar E)$ and
  $\mathcal{L}_0(K,E)$ are equal.
\end{proof}

We now have all the tools to prove theorem \ref{thm:mild2}.
\begin{proof}[Proof of theorem \ref{thm:mild2}]
  As a first step we prove that the mapping $\F:
  \mathbb{H}_p(T) \to \mathbb{H}_p(T)$ defined by
  \begin{equation*}
    \F u(t) =  e^{tA} u_0 
    + \int_0^t e^{(t-s)A} f(s,u(s))\,ds
    + \int_0^t e^{(t-s)A} B(s,u(s))\,dM(s)
  \end{equation*}
  is well defined and is a contraction.  Let us prove that the image
  of $\F$ is in fact contained in $\mathbb{H}_p(T)$.
  Predictability of $\F u$ follows as in the proof of theorem
  \ref{thm:mild1}. Let us prove that $\|\F u\|_p<\infty$:
  Minkowski's inequality yields
  \begin{eqnarray*}
    \|\F u\|^p_p &\lesssim_p& \E\sup_{t\leq T}|e^{tA} u_0|^p
    + \E\sup_{t\leq T}\Big| \int_0^t e^{(t-s)A} f(s,u(s))\,ds \Big|^p\\
    && + \E\sup_{t\leq T}\Big| \int_0^t e^{(t-s)A} B(s,u(s))\,dM(s) \Big|^p,
  \end{eqnarray*}
  where $\E\sup_{t\leq T}|e^{tA}u_0|^p \leq \sup_{t\leq T} |e^{tA}|^p
  \E|u_0|^p<\infty$.  Moreover, by Cauchy-Schwarz' inequality and
  (\ref{eq:i}) we get
  \begin{eqnarray*}
    \E\sup_{t\leq T} \left|\int_0^t e^{(t-s)A}f(s,u(s))\,ds\right|^p &\leq&
    \E\sup_{t\leq T} \left[\int_0^t |e^{(t-s)A}f(s,u(s))|\,ds\right]^p \\
    &\leq& T^{p/2} \E\sup_{t\leq T}\left[\int_0^t |e^{(t-s)A}f(s,u(s))|^2\,ds\right]^{p/2} \\
    &\leq& T^{p/2} \sup_{t\leq T} \Big(\int_0^t h(t-s)^2\,ds\Big)^{p/2}
    \E\sup_{t\leq T} (1+|u(t)|)^p
  \end{eqnarray*}
  hence
  $$
  \E \sup_{t\leq T} \left|\int_0^t e^{(t-s)A}f(s,u(s))\,ds\right|^p
  \lesssim (1+ \E\sup_{t\leq T}|u(t)|^p) < \infty.
  $$

  Similarly, using the maximal inequality (\ref{eq:BDGconv}) and
  (\ref{eq:ii}), we obtain
  \begin{align*}
    \E\sup_{t\leq T} \left| \int_0^t e^{(t-s)A}B(s,u(s))\,dM(s) \right|^p
    &\lesssim_{p,T,m} \, \E\!\int_0^T |B(s,u(s))|^p\,ds \\
    &\leq \, \E\!\int_0^T h(s)^p (1+|u(s)|)^p\,ds\\
    &\lesssim  \big(1 + \E\sup_{s\leq T} |u(t)|^p \big) \int_0^T h(s)^p\,ds
    < \infty,
  \end{align*}
  thus completing the proof that $\F(\mathbb{H}_p(T))
  \subset \mathbb{H}_p(T)$.

  Analogous calculations show that there exists a constant
  $N=N(T,h,p,m)$ such that $\|\F u - \F v\|_p \leq N\|u-v\|_p$, with
  $N(T_0,h,p,m)<1$ for some $T_0>0$. The existence and uniqueness
  proof is then finished exactly as in the proof of theorem
  \ref{thm:mild1}. The estimate (\ref{eq:fuoco}) is also established
  in the same way, with very minor modifications.
\end{proof}

\section{Proofs of theorems \ref{thm:1} and \ref{thm:2}}
\label{sec:proofs}
In order to apply theorems \ref{thm:mild1} and \ref{thm:mild2} to
(\ref{eq:musso}) we need to prove that $B$ and $f$ satisfy Lipschitz
conditions. Unfortunately we can only prove that $B$ and $f$ are locally
Lipschitz, hence we obtain only a local existence and uniqueness
result.

Let us briefly recall that the volatility coefficient $\sigma$
appearing in (\ref{eq:musso}) is a random function from
$[0,T]\times\erre_+\times\erre$ to $K$, and that the operator $B$ is
defined as $[B(t,u)\phi](x) = \ip{\sigma(t,x,u(x))}{\phi}_K$, for
$\phi \in K$.
Let us also define the space $\fH=L_{2,\alpha}^1(\erre_+,K)$, equipped
with the norm
$$
|\phi|_\fH^2 := 
|\phi(0)|_K^2 + \int_0^\infty |\phi'(x)|_K^2 \alpha(x)\,dx < \infty.
$$

\begin{prop}\label{prop:locl1}
Assume that hypotheses (i)--(iv) of theorem \ref{thm:1} are
satisfied.  Then $B(t,u) \in \L_2(K,H)$ for all $u$ and there exists a
constant $N$ such that $|B(t,u)|_2 \leq N(1+R)|u|$ and
$|B(t,u)-B(t,v)|_2 \leq N(1+R)|u-v|_H$ for all $u$, $v \in B_R(H)$,
uniformly over $t\geq 0$.
\end{prop}
\begin{proof}
  We have
  $$
  |B(t,u)|_2^2 = \sum_{k=1}^\infty | B(t,u)e_k |_H^2
  = \sum_{k=1}^\infty \int_0^\infty |\sigma^k(t,x,u(x))_x|^2 \alpha(x)\,dx
  $$
  and
  \begin{equation*}
    \int_0^\infty |\sigma^k(t,x,u(x))_x|^2 \alpha(x)\,dx
    \lesssim \int_0^\infty \big(\sigma^k_x(t,x,u(x))^2
    + \sigma^k_u(t,x,u(x))^2 u'(x)^2\big) \alpha(x)\,dx,
  \end{equation*}
  where, using (\ref{eq:gna}) and (\ref{eq:pippo-inf}),
  \begin{align*}
    \int_0^\infty |\sigma^k_x(t,x,u(x))|^2 \alpha(x)\,dx &\lesssim
    \int_0^\infty |\sigma^k_x(t,x,u(x))-\sigma^k_x(t,x,0)|^2 \alpha(x)\,dx
    + |\sigma(t,\cdot,0)|^2_\fH\\
    &\leq \int_0^\infty \beta^k(x)^2 u(x)^2 \alpha(x)\,dx
    + |\sigma(t,\cdot,0)|^2_\fH\\
    &\leq |\beta^k|^2_\alpha |u|_\infty^2 + |\sigma(t,\cdot,0)|^2_\fH
    \lesssim_\alpha |\beta^k|^2_\alpha |u|_H^2 + |\sigma(t,\cdot,0)|^2_\fH,
  \end{align*}
hence
\begin{equation}
\label{eq:buu}
|B(t,u)|_2^2 \leq |\beta|^2_{\alpha,\ell_2} |u|_H^2 +
|\sigma(t,\cdot,0)|^2_\fH < \infty.
\end{equation}
Similarly,
\begin{equation*}
|B(t,u)-B(t,v)|_2^2 = \sum_{k=1}^\infty \int_0^\infty
\big|\sigma^k(t,x,u(x))_x - \sigma^k(t,x,v(x))_x\big|^2 \alpha(x)\,dx,
\end{equation*}
and
\begin{multline*}
\big|(\sigma^k(t,x,u(x)) - \sigma^k(t,x,v(x)))_x \big|^2 \lesssim
\big|\sigma^k_x(t,x,u(x)) - \sigma^k_x(t,x,v(x)) \big|^2\\
+ \big|\sigma^k_u(t,x,u(x))u'(x) - \sigma^k_u(t,x,v(x))v'(x)\big|^2.
\end{multline*}
We also have 
\begin{equation*}
\sum_{k=1}^\infty \int_0^\infty
\big|\sigma^k_x(t,x,u(x)) - \sigma^k_x(t,x,v(x)) \big|^2 \alpha(x)\,dx
\leq |\beta|^2_{\alpha,\ell_2} \, |u-v|_H^2,
\end{equation*}
and, by virtue of (\ref{eq:pru}) and (\ref{eq:pippo-inf}),
\begin{eqnarray*}
\lefteqn{\sum_{k=1}^\infty \int_0^\infty \big|\sigma^k_u(t,x,u(x))u'(x)
   - \sigma^k_u(t,x,v(x))v'(x)\big|^2 \alpha(x)\,dx}\\
&\lesssim& \sum_{k=1}^\infty \int_0^\infty \big|\sigma^k_u(t,x,u(x))u'(x)
    - \sigma^k_u(t,x,u(x))v'(x)\big|^2 \alpha(x)\,dx\\
&& + \sum_{k=1}^\infty \int_0^\infty \big|\sigma^k_u(t,x,u(x))v'(x)
     - \sigma^k_u(t,x,v(x))v'(x)\big|^2 \alpha(x)\,dx \\
&\leq& |\gamma|^2_{\ell_2} |u-v|^2_{1,\alpha}
       + |\gamma|^2_{\ell_2} \int_0^\infty |(u(x)-v(x))v'(x)|^2 \alpha(x)\,dx\\
&\lesssim& |u-v|^2_H + |u-v|^2_\infty |v|^2_{1,\alpha}
       \lesssim (1+|v|_H^2)|u-v|^2_H,
\end{eqnarray*}
where in the second last step we have used the fact that
$|\sigma^k_{uu}|<\gamma^k$ implies that $\sigma^k_u$ is Lipschitz with
respect to $u$ with Lipschitz constant not greater than $\gamma^k$.
The above estimates imply that $|B(t,u)-B(t,v)|_2 \lesssim
(1+R)|u-v|_H$ for all $|u|_H$, $|v|_H\leq R$, hence also, by
(\ref{eq:buu}),
\begin{equation*}
|B(t,u)|_2 \lesssim |B(t,0)|_2 + (1+R)|u|
\leq |\sigma(t,\cdot,0)|_\fH + (1+R)|u|
\lesssim (1+R)|u|.
\qedhere
\end{equation*}
\end{proof}

\medskip

We shall now obtain conditions under which the function $g$ defined in
(\ref{eq:g}) satisfies a local Lipschitz condition. A useful tool is
given by the following lemma, which essentially allows to obtain
vector-valued versions of the inequalities
(\ref{eq:pippo-inf})-(\ref{eq:pippo-4}).

\begin{lemma}\label{lem:limo}
Let $\phi \in \fH$.
Then $|\phi|_K \in H$ and $||\phi|_K|_H \leq |\phi|_\fH$.
\end{lemma}
\begin{proof}
  It is enough to prove the statement assuming $\phi(0)=0$, as it is
  immediate to see by comparing the definitions of $||\cdot|_K|_H$ and
  $|\cdot|_\fH$.  Let $\eta_\varepsilon \in C^1(K,\erre)$ be
  a smooth approximation of the norm of $K$, such that
  $|D\eta_\varepsilon(x)|\leq 1$ for all $x\in K$.  Then we have
  \begin{eqnarray*}
    |\eta_\varepsilon(\phi)|^2_H &=&
    \int_0^\infty |D_x \eta_\varepsilon(\phi(x))|_K^2\alpha(x)\,dx\\
    &=& \int_0^\infty |D\eta_\varepsilon(\phi(x))|_K^2
              |\phi'(x)|^2\alpha(x)\,dx\\
    &\leq& \int_0^\infty |\phi'(x)|_K^2\alpha(x)\,dx
    = |\phi|^2_\fH < \infty.
  \end{eqnarray*}
  Since the above bound does not depend on $\varepsilon$ and
  $\eta_\varepsilon(\phi(x)) \to |\phi(x)|_K$ as $\varepsilon \to 0$,
  we infer that $|\phi(x)|_K \in H$ and $||\phi|_K|_H \leq
  |\phi|_\fH$.
\end{proof}

\begin{prop}\label{prop:locl2}
  There exists a positive constant $N$ depending only on $\alpha$ and
  $\psi$ such that
  $$
  |g(\sigma) - g(\rho)|_H \leq N (1+R^2) |\sigma-\rho|_\fH
  $$
  for all $\sigma$, $\rho \in B_R(\fH)$.
\end{prop}
\begin{proof}
By definition of $g$ one immediately gets
$$
|g(\sigma) - g(\rho)|_H = |g(\sigma(0)) - g(\rho(0))|
+ |g(\sigma)-g(\rho)|_{1,\alpha}
$$
and $|g(\sigma(0)) - g(\rho(0))| \lesssim_\psi |\sigma(0)-\rho(0)|_K
\leq |\sigma-\rho|_\fH$.  Setting $S(x)=-\int_0^x\sigma(y)\,dy$ and
$R(x)=-\int_0^x\rho(y)\,dy$, a simple calculation reveals that
\begin{eqnarray*}
\lefteqn{|g(\sigma)-g(\rho)|^2_{1,\alpha}}\\
&\lesssim& \int_0^\infty \big| \ip{\sigma'(x)}{D\psi(S(x))}
- \ip{\rho'(x)}{D\psi(R(x))} \big|^2 \alpha(x)\,dx\\
&& +\int_0^\infty \big| \ip{\sigma(x)}{D^2\psi(S(x))\sigma(x)}
- \ip{\rho(x)}{D^2\psi(R(x))\rho(x)} \big|^2 \alpha(x)\,dx.
\end{eqnarray*}
Since $\psi\in C^3_b(B_r(K))$, in particular $D\psi$ is Lipschitz and
bounded on $B_r(K)$, we have
\everymath{\displaystyle}
\begin{equation*}
\begin{array}{rll}
\big| \ip{\sigma'}{D\psi(S)} - \ip{\rho'}{D\psi(R)}\big|^2_\alpha &\lesssim&
\big| \ip{\sigma'}{D\psi(S)}_K - \ip{\rho'}{D\psi(S)}_K \big|^2_\alpha\\
&& +  \big| \ip{\rho'}{D\psi(S)}_K - \ip{\rho'}{D\psi(R)}_K \big|^2_\alpha\\
&\lesssim_\psi& |\sigma-\rho|^2_\fH + |\rho|^2_\fH
         \Big(\int_0^\infty |\sigma(x)-\rho(x)|_K\,dx\Big)^2\\
&\lesssim_\alpha& |\sigma-\rho|^2_\fH + |\rho|^2_\fH\,
               ||\sigma-\rho|_K|_H^2\\
&\leq& (1+|\rho|^2_\fH)|\sigma-\rho|^2_\fH
\end{array}
\end{equation*}
where the second last and last inequalities follow by (\ref{eq:pippo-1}) and
lemma \ref{lem:limo}, respectively.
\par\noindent
Similarly we have
\begin{eqnarray*}
\lefteqn{\big| \ip{\sigma}{D^2\psi(S)\sigma} 
- \ip{\rho}{D^2\psi(R)\rho} \big|^2_\alpha}\\
&\lesssim&  \big| \ip{\sigma}{D^2\psi(S)\sigma} 
- \ip{\rho}{D^2\psi(S)\rho} \big|^2_\alpha
+  \big| \ip{\rho}{(D^2\psi(S)-D^2\psi(R))\rho} \big|^2_\alpha.
\end{eqnarray*}
Note that for any operator $Q\in \mathcal{L}(K)$ we have
\begin{eqnarray*}
  |\ip{Qx}{x} - \ip{Qy}{y}|^2 &\leq& 2 |\ip{Qx}{x} - \ip{Qx}{y}|^2
  + 2 |\ip{Qx}{y} - \ip{Qy}{y}|^2\\
&\leq& 2|Qx|^2|x-y|^2 + 2|Q(x-y)|^2|y|^2\\
&\leq& 2|Q|^2(|x|^2+|y|^2)|x-y|^2,
\end{eqnarray*}
therefore, by Cauchy-Schwarz' inequality, (\ref{eq:pippo-4}), and lemma
\ref{lem:limo},
\begin{align*}
\lefteqn{\big| \ip{\sigma}{D^2\psi(S)\sigma} 
- \ip{\rho}{D^2\psi(S)\rho} \big|^2_\alpha} \\
&\lesssim_\psi \int_0^\infty (|\sigma(x)|_K^2 + |\rho(x)|_K^2)
        |\sigma(x)-\rho(x)|_K^2 \alpha(x)\,dx\\
&\lesssim \Big( \int_0^\infty (|\sigma(x)|_K^4 + |\rho(x)|_K^4) \alpha(x)\,dx
 \Big)^{1/2}
   \Big( \int_0^\infty |\sigma(x) -\rho(x)|_K^4) \alpha(x)\,dx \Big)^{1/2}\\
&\lesssim_\alpha (||\sigma|_K|_H^2 + ||\rho|_K|_H^2) \, ||\sigma-\rho|_K|_H^2
   \leq (|\sigma|_\fH^2 + |\sigma|_\fH^2) \,
        |\sigma-\rho|_\fH^2.
\end{align*}
Finally, since $D^2\psi$ is Lipschitz, applying (\ref{eq:pippo-1}),
(\ref{eq:pippo-4}) and lemma \ref{lem:limo}, we get
\begin{eqnarray*}
  |\ip{\rho}{(D^2\psi(S)-D^2\psi(R))\rho}|^2_\alpha &\lesssim_\psi&
  ||\rho|_K^2|_\alpha^2 \Big(\int_0^\infty |\sigma(x)-\rho(x)|_K\,dx\Big)^2\\
  &\lesssim_\alpha& |\rho|_\fH^4 \, ||\sigma-\rho|_K|_H^2
     \leq |\rho|_\fH^4 \, |\sigma-\rho|_\fH^2,
\end{eqnarray*}
from which the claim follows.
\end{proof}


\begin{proof}[Proof of theorem \ref{thm:1}]
  Let $u$, $v \in B_R(H)$. Then we have $|\sigma(t,u)-\sigma(t,v)|_\fH
  \leq N(1+R)|u-v|_H$. In fact, obvious steps show that
  \begin{align*}
  |\sigma(t,u)-\sigma(t,v)|^2_\fH &=
  \int_0^\infty |\sigma(t,x,u(x))_x-\sigma(t,x,v(x))_x|_K^2 \alpha(x)\,dx\\
  &\leq |B(t,u)-B(t,v)|_2^2 \leq N(1+R)^2 |u-v|^2_H,
  \end{align*}
  where the last inequality follows by proposition \ref{prop:locl1}.
  This estimate, together with proposition \ref{prop:locl2}, implies
  that there exists a constant $N$ such that $|f(t,u)-f(t,v)|_H \leq
  N(1+R^2)|u-v|_H$ for all $u$, $v\in B_R(H)$. In turns, this implies
  that $f$ is locally bounded, as a consequence of hypothesis (iii).
  Similarly, $B$ is locally bounded, as follows by (\ref{eq:buu}). The
  proof is then finished in a standard way appealing to theorem
  \ref{thm:mild1}.
\end{proof}

\begin{proof}[Proof of theorem \ref{thm:2}]
Let us denote by $\dot{H}$ the space $L_{2,\alpha}^1$ endowed with the norm
$$
|\phi|^2_* = \phi(\infty)^2 + \int_0^\infty \phi'(x)^2 \alpha(x)\,dx,
$$ 
which is equivalent to the norm of $H$. Moreover, define as
$\dot{H}_0$ the subspace of functions $\phi \in \dot{H}$ such that
$\phi(\infty)=0$. Then the semigroup $e^{tA}$ is a contraction in
$\dot{H}_0$, because
$$
|e^{tA}\phi|_*^2 = \int_t^\infty \phi'(x)^2 \alpha(x)\,dx
\leq |\phi|_H^2.
$$ 
Therefore we can apply theorem \ref{thm:mild2}, noting that the
necessary properties of $f$ and $B$ are proved as in the previous
proof. We thus obtain the existence and uniqueness of a solution $u$
with values in $\dot{H}$. But since the norms of $H$ and $\dot{H}$ are
equivalent, $u$ is also well defined as a process in
$\mathbb{H}_p(T)$.
\end{proof}

\begin{rmk}
  Note that instead of starting from equation (\ref{eq:musso}), one
  could start directly from an abstract evolution equation like
  (\ref{eq:ee}), with $B$ an operator depending on the whole forward
  curve $u(t,\cdot)$. In fact, roughly speaking, (\ref{eq:musso}) is
  just a special case of (\ref{eq:ee}) where $B$ is the Nemitski
  operator associated to $\sigma$. Unfortunately we have only been
  able to prove that such $B$ is locally Lipschitz with respect
  to $u$. Nothing forbids to assume that, in the general case, $B$ is
  Lipschitz, but even so we could only prove that $f$ is locally
  Lipschitz, hence mild solutions would still be only local.
\end{rmk}

\section{Examples}\label{sec:ex}
\subsection{Finite dimensional noise}
Let us consider, for simplicity, the case $K = \erre$ (the more
general case $K=\erre^d$, $d<\infty$, being similar). Let $M$ be a
real valued L\'evy process such that $M(1)$ admits an analytic
characteristic function, e.g. a jump-diffusion (i.e. the sum of a
Wiener process and a Poisson process) or a Gamma process (which are
discussed, for instance, in \cite{ALM}). Then, thanks to the
analyticity of the characteristic function of $M(1)$, we infer that
$\psi \in C^\infty(B_r)$ (see e.g. \cite{lukacs}). In particular,
since $B_r$ is compact, $\psi \in C^3_b(B_r)$, and all the hypotheses
on $M$ of section \ref{sec:mr} are satisfied.

\subsection{Infinitely many independent noise sources}
Consider a model of the type (\ref{eq:musso}), where $\sigma=\sum_k
\sigma^ke_k$ and $M$ is formally defined as
\begin{equation}\label{eq:formale}
M(t) = \sum_{k=1}^\infty \xi^k(t)e_k,
\end{equation}
where $\xi^k$, $k\in\enne$, are real independent L\'evy processes.  In
order to consider the problem in the setting developed above, we need
to establish conditions under which $M$ is a well defined L\'evy
process on the Hilbert space $K$.
\begin{lemma}\label{lem:formale}
  Let $\xi^k$, $k\in\enne$, be real independent L\'evy processes
  with characteristic triplets $(b^k,r^k,m^k)$, $k\in\enne$. If
  $b^\cdot$, $r^\cdot \in \ell_2$ and
$$
\sum_{k=1}^\infty \int_\erre (1\wedge x^2)\,m^k(dx) < \infty,
$$
then (\ref{eq:formale}) defines a L\'evy process on $K$ with
characteristic triplet $(b,R,m)$, where $b=\sum_{k=1}^\infty b^ke_k$,
$R: K \ni y \mapsto \sum r^k\ip{y}{e_k}e_k$, and $m:\mathcal{B}(K)\ni
A \mapsto \sum m^k(A^k)$, where $A^k$ denotes the projection of $A$ on
$\mathrm{span}(e_k)$.
\end{lemma}
\begin{proof}
One has, using the summation convention over repeated indices (when
no confusion arises),
$$
  \E e^{i\ip{a}{M(t)}} = \E e^{i\ip{a^ke_k}{\xi_t^k e_k}} =
  \E e^{ia_k\xi_t^k} = e^{-t \sum \lambda^k(a^k)},
$$
and
\begin{eqnarray*}
\lambda(a) &:=& (\sum\lambda^k)(a^k) = -ib^ka^k + \frac12 r^k(a^k)^2
-\sum \int_\erre \Big(e^{ia^kx} - 1 - i\frac{a^kx}{1+x^2}\Big)\,m^k(dx) \\
&=& -i\ip{b}{a} + \frac12\ip{Ra}{a}
- \int_H \Big(e^{i\ip{a}{x}} - 1 - i\frac{\ip{a}{x}}{1+|x|^2}\Big)\,m(dx),
\end{eqnarray*}
i.e. $\E e^{i\ip{a}{M(t)}}=e^{-t \lambda(a)}$, which is equivalent to
the claim of the lemma.
\end{proof}

Let us now consider the properties of the function $\psi$. We have
\begin{eqnarray*}
\psi(z) &=& \log \int_K e^{\ip{x}{z}}\,\mu(dx)
= \log \prod_{k=1}^\infty \int_\erre e^{xz^k}\,\mu^k(dx) \\
&=& \sum _{k=1}^\infty \log \int_\erre e^{xz^k}\,\mu^k(dx),
\end{eqnarray*}
where $\mu$ is the law of $M(1)$ and $\mu^k$ is the law of $\xi^k(1)$,
for all $k\in \enne$.  Let us define the function $\varphi:B_{\delta
  r}(\ell_2) \to \erre^\infty$ as
$$
\varphi: z^\cdot \mapsto \log \int_\erre e^{xz^k}\,\mu^k(dx).
$$
If the measures $\mu^k$ are such that the image of $B_{\delta
  r}(\ell_2)$ under $\varphi$ is contained in $\ell_1$, then
(\ref{eq:mah}) is satisfied. In particular, this is true if
$$
k \mapsto \log \int_{\erre\setminus[-1,1]} e^{|z^k|\,|x|}\,m^k(dx) \in \ell_1
$$
for all $z^\cdot \in B_{\delta r}(\ell_2)$, as it follows by theorem 25.3 of
\cite{sato}.
Consider for instance infinitely many independent Gamma processes
$(\xi^k)_{k\in\enne}$. Then the L\'evy measure of $\mu^k$ is given by
$$
m^k(dx) = c^k x^{-1} e^{-\alpha^k x}\,dx, \qquad x>0,
$$
where $c^k$, $\alpha^k$ are positive numbers. In the L\'evy-Kintchine
representation of $\mu^k$ one has $r^k=0$ and
$b^k=c^k(\alpha^k)^{-1}$. Therefore
$$
\sum_{k=1}^\infty \int_0^\infty (1\wedge x^2)\,m^k(dx)
\leq \sum_{k=1}^\infty \int_0^\infty x^2 \,m^k(dx),
$$
and the last integral is finite if
$$
\sum_{k=1}^\infty \int_0^\infty x^2 \,\mu^k(dx) =
\sum_{k=1}^\infty (\alpha^k)^{-2} (c^k + (c^k)^2).
$$
Therefore, assuming that $(\alpha^k)^{-1} \in \ell_\infty$ and $c^k
\in \ell_1 \cap \ell_2$, lemma \ref{lem:formale} implies that the
series $\sum_k \xi^k(t)e_k$ defines a L\'evy process on the Hilbert
space $K$.  We also have
$$
\log \int_\erre e^{xz^k}\,\mu_k(dx) =
- c^k \log(1-(\alpha^k)^{-1}z_k),
\qquad z^k < \alpha^k,
$$
hence the $\ell_1$ norm of the left-hand side is bounded from above by
$$
|c|_{\ell_1} \big|\log(1-|1/\alpha|_{\ell_\infty} \delta r)\big|,
$$
where $1/\alpha$ stands for the sequence $(1/\alpha^k)_{k\in\enne}$.
This implies that choosing $\delta$ and $r$ so that $\delta r <
|1/\alpha|_{\ell_\infty}^{-1}$, we obtain the local well-posedness of
an HJM model driven by infinitely many Gamma processes, as follows by
the results of section 2.

\section{Conclusions}
We have proved existence and uniqueness of local mild solutions to
Musiela's SPDE in a Hilbert space $H$ of absolutely continuous
functions, that seems to be the standard ``ambient space'' for forward
curves. While this space has several nice features, which are
extensively discussed in \cite{filipo}, its main drawback is probably
that it does not allow one to characterize forward curves at a given
time $T>0$ as solutions of an SDE in $H$ (or at least we failed to
find sufficient conditions). It is known that if $M$ is a Wiener
process, then a unique global solution exists in weighted $L_2$ spaces
(see e.g. \cite{GM-review}). On the other hand, such spaces are too
big, in the sense that very rough (with respect to $x$) forward curves
are allowed, and this should be ruled out in any sensible model for
the dynamics of forward rates (see also the discussion in
\cite{EkeTaf}). Therefore, it seems that one should find the ``right''
space to study Musiela's SPDE, satisfying the minimum requirement that
its elements admit a continuous modification, and allowing at the same
time to obtain existence and uniqueness of global mild solutions. This
problem, to the best of our knowledge, is not solved also in the case
of Brownian noise.

\subsection*{Acknowledgments}
This work was partially supported by the DFG through the SFB 611,
Bonn, and by the ESF through grant AMaMeF 969. This work was carried
out while the author was visiting the Max-Planck-Institut f\"ur
Mathematik in Leipzig supported by an EPDI fellowship. The author is
sincerely grateful to S.~Albeverio for helpful discussions on some
parts of the paper, to two anonymous referees and to E.~Eberlein for
suggestions which led to improvements of the presentation.

\let\oldbibliography\thebibliography
\renewcommand{\thebibliography}[1]{%
  \oldbibliography{#1}%
  \setlength{\itemsep}{-1pt}%
}

\bibliographystyle{amsplain}
\bibliography{ref}

\def\polhk#1{\setbox0=\hbox{#1}{\ooalign{\hidewidth
  \lower1.5ex\hbox{`}\hidewidth\crcr\unhbox0}}}
\providecommand{\bysame}{\leavevmode\hbox to3em{\hrulefill}\thinspace}
\providecommand{\MR}{\relax\ifhmode\unskip\space\fi MR }
\providecommand{\MRhref}[2]{%
  \href{http://www.ams.org/mathscinet-getitem?mr=#1}{#2}
}
\providecommand{\href}[2]{#2}
\begin{thebibliography}{10}

\bibitem{ALM}
S.~Albeverio, E.~Lytvynov, and A.~Mahnig, \emph{A model of the term structure
  of interest rates based on {L}\'evy fields}, Stochastic Process. Appl.
  \textbf{114} (2004), no.~2, 251--263. \MR{MR2101243 (2005h:60178)}

\bibitem{AMR2}
S.~Albeverio, V.~Mandrekar, and B.~R{\"u}diger, \emph{Existence of mild
  solutions for stochastic differential equations and semilinear equations with
  non-{G}aussian {L}\'evy noise}, Preprint SFB 611, Bonn, 2006.

\bibitem{AmbPro}
A.~Ambrosetti and G.~Prodi, \emph{A primer of nonlinear analysis}, Cambridge
  University Press, Cambridge, 1995. \MR{MR1336591 (96a:58019)}

\bibitem{BGJ}
K.~Bichteler, J.-B. Gravereaux, and J.~Jacod, \emph{Malliavin calculus for
  processes with jumps}, Gordon and Breach Science Publishers, New York, 1987.
  \MR{MR1008471 (90h:60056)}

\bibitem{BDMKR}
T.~Bj{\"o}rk, G.~Di~Masi, Yu. Kabanov, and W.~Runggaldier, \emph{Towards a
  general theory of bond markets}, Finance Stochast. \textbf{1} (1997),
  141--174.

\bibitem{DZ92}
G.~Da~Prato and J.~Zabczyk, \emph{Stochastic equations in infinite dimensions},
  Cambridge UP, 1992.

\bibitem{EbeJacRai}
E.~Eberlein, J.~Jacod, and S.~Raible, \emph{L\'evy term structure models:
  no-arbitrage and completeness}, Finance Stoch. \textbf{9} (2005), no.~1,
  67--88. \MR{MR2210928}

\bibitem{EbeOez}
E.~Eberlein and F.~{\"O}zkan, \emph{The defaultable {L}\'evy term structure:
  ratings and restructuring}, Math. Finance \textbf{13} (2003), no.~2,
  277--300. \MR{MR1967777 (2004a:91065)}

\bibitem{EbeRai}
E.~Eberlein and S.~Raible, \emph{Term structure models driven by general
  {L}\'evy processes}, Math. Finance \textbf{9} (1999), no.~1, 31--53.
  \MR{MR1849355 (2002f:91039)}

\bibitem{EkeTaf}
I.~Ekeland and E.~Taflin, \emph{A theory of bond portfolios}, Ann. Appl.
  Probab. \textbf{15} (2005), no.~2, 1260--1305. \MR{MR2134104 (2006b:91067)}

\bibitem{filipo}
D.~Filipovi{\'c}, \emph{Consistency problems for {H}eath-{J}arrow-{M}orton
  interest rate models}, Lecture Notes in Mathematics, vol. 1760,
  Springer-Verlag, Berlin, 2001. \MR{MR1828523 (2002e:91001)}

\bibitem{FilTap}
D.~Filipovi{\'c} and S.~Tappe, \emph{Existence of {L\'e}vy term structure
  models}, preprint. To appear in \emph{Finance Stoch.}, 2007.

\bibitem{GM-review}
B.~Goldys and M.~Musiela, \emph{Infinite dimensional diffusions, {K}olmogorov
  equations and interest rate models}, Option pricing, interest rates and risk
  management, Handb. Math. Finance, Cambridge Univ. Press, Cambridge, 2001,
  pp.~314--335. \MR{1 848 556}

\bibitem{Gyo-semimg}
I.~Gy{\"o}ngy, \emph{On stochastic equations with respect to semimartingales.
  {III}}, Stochastics \textbf{7} (1982), no.~4, 231--254.

\bibitem{HZ}
H.~D. Hamedani and B.~Z. Zangeneh, \emph{Stopped {D}oob inequality for {$p$}-th
  moment, {$0<p<\infty$}, stochastic convolution integrals}, Stochastic Anal.
  Appl. \textbf{19} (2001), no.~5, 771--798. \MR{MR1857896 (2002h:60120)}

\bibitem{HauSei}
E.~Hausenblas and J.~Seidler, \emph{A note on maximal inequality for stochastic
  convolutions}, Czechoslovak Math. J. \textbf{51(126)} (2001), no.~4,
  785--790. \MR{MR1864042 (2002j:60092)}

\bibitem{Ichi}
A.~Ichikawa, \emph{Some inequalities for martingales and stochastic
  convolutions}, Stochastic Anal. Appl. \textbf{4} (1986), no.~3, 329--339.
  \MR{MR857085 (87m:60105)}

\bibitem{JZ}
J.~Jakubowski and J.~Zabczyk, \emph{Exponential moments for {HJM} models with
  jumps}, Finance Stoch. \textbf{11} (2007), no.~3, 429--445. \MR{MR2322920}

\bibitem{Knoche-CRAS}
C.~Knoche, \emph{S{PDE}s in infinite dimension with {P}oisson noise}, C. R.
  Math. Acad. Sci. Paris \textbf{339} (2004), no.~9, 647--652. \MR{MR2103204}

\bibitem{Kote-sub}
P.~Kotelenez, \emph{A submartingale type inequality with applications to
  stochastic evolution equations}, Stochastics \textbf{8} (1982/83), no.~2,
  139--151.

\bibitem{Kote-Doob}
\bysame, \emph{A stopped {D}oob inequality for stochastic convolution integrals
  and stochastic evolution equations}, Stochastic Anal. Appl. \textbf{2}
  (1984), no.~3, 245--265.

\bibitem{LR-heat}
P.~Lescot and M.~R{\"o}ckner, \emph{Perturbations of generalized {M}ehler
  semigroups and applications to stochastic heat equations with {L\'e}vy noise
  and singular drift}, Potential Anal. \textbf{20} (2004), no.~4, 317--344.

\bibitem{lukacs}
E.~Lukacs, \emph{Characteristic functions}, Hafner Publishing Co., New York,
  1970. \MR{MR0346874 (49 \#11595)}

\bibitem{CM-musdet}
C.~Marinelli, \emph{Well-posedness and invariant measures for {HJM} models with
  deterministic volatility and {L\'evy} noise}, preprint arXiv:math.PR/0702622,
  2007.

\bibitem{Met}
M.~M{\'e}tivier, \emph{Semimartingales}, Walter de Gruyter \& Co., Berlin,
  1982. \MR{MR688144 (84i:60002)}

\bibitem{MP-Z}
M.~M{\'e}tivier and G.~Pistone, \emph{Une formule d'isom\'etrie pour
  l'int\'egrale stochastique hilbertienne et \'equations d'\'evolution
  lin\'eaires stochastiques}, Z. Wahrscheinlichkeitstheorie und Verw. Gebiete
  \textbf{33} (1975/76), no.~1, 1--18. \MR{MR0383527 (52 \#4408)}

\bibitem{MP}
\bysame, \emph{Sur une \'equation d'\'evolution stochastique}, Bull. Soc. Math.
  France \textbf{104} (1976), no.~1, 65--85. \MR{MR0420854 (54 \#8866)}

\bibitem{PZ-levico}
Sz. Peszat and J.~Zabczyk, \emph{Stochastic heat and wave equations driven by
  an impulsive noise}, Stochastic partial differential equations and
  applications---VII, Lect. Notes Pure Appl. Math., vol. 245, Chapman \&
  Hall/CRC, Boca Raton, FL, 2006, pp.~229--242. \MR{MR2227232}

\bibitem{PZ07}
\bysame, \emph{Heath-{J}arrow-{M}orton-{M}usiela equation of bond market},
  IMPAN preprint, 2007.

\bibitem{ProTal-Euler}
Ph. Protter and D.~Talay, \emph{The {E}uler scheme for {L}\'evy driven
  stochastic differential equations}, Ann. Probab. \textbf{25} (1997), no.~1,
  393--423. \MR{MR1428514 (98c:60063)}

\bibitem{Protter}
Ph.~E. Protter, \emph{Stochastic integration and differential equations},
  second ed., Springer-Verlag, Berlin, 2004. \MR{MR2020294 (2005k:60008)}

\bibitem{sato}
Ken-iti Sato, \emph{L\'evy processes and infinitely divisible distributions},
  Cambridge University Press, Cambridge, 1999. \MR{MR1739520 (2003b:60064)}

\bibitem{tehranchi}
M.~Tehranchi, \emph{A note on invariant measures for {HJM} models}, Finance
  Stoch. \textbf{9} (2005), no.~3, 389--398. \MR{MR2211714}

\bibitem{vargiolu}
T.~Vargiolu, \emph{Invariant measures for the {M}usiela equation with
  deterministic diffusion term}, Finance Stoch. \textbf{3} (1999), no.~4,
  483--492. \MR{2002i:60080}

\end{thebibliography}

\end{document}